\documentclass[a4paper,11pt]{amsart}
\usepackage{amsfonts,amsthm,amsmath,amssymb,graphicx}
\theoremstyle{plain}
\newtheorem{lem}{Lemma}[section]
\newtheorem{prop}[lem]{Proposition}
\newtheorem{thm}[lem]{Theorem}
\newtheorem{cor}[lem]{Corollary}
\theoremstyle{definition}

\theoremstyle{remark}

\DeclareMathOperator{\sgn}{sgn}
\DeclareMathOperator{\cls}{cls}
\DeclareMathOperator{\rank}{rank}

\DeclareMathOperator{\mass}{mass}
\DeclareMathOperator{\ord}{ord}
\DeclareMathOperator{\sym}{sym}

\DeclareMathOperator{\e}{e}

\DeclareMathOperator{\diag}{diag}
\DeclareMathOperator{\gen}{gen}

\newcommand{\Z}{\mathbb Z}
\newcommand{\Q}{\mathbb Q}
\newcommand{\A}{\mathbb A}

\newcommand{\E}{\mathbb E}
\newcommand{\R}{\mathbb R}
\newcommand{\C}{\mathbb C}
\newcommand{\stufe}{\mathcal N }
\newcommand{\G}{\mathcal G}

\newcommand{\h}{\mathfrak H}

\newcommand{\M}{\mathcal M}
\newcommand{\Y}{\mathcal Y}

\newcommand{\calL}{\mathcal L}

\begin{document}

% Enter full title and short title for running headers
\title[Average Siegel theta series as Eisenstein series]
{Explicitly realizing average Siegel theta series as linear combinations of Eisenstein series}
%\shorttitle{Degree 2 Eisenstein series}

% Author name(s)
\author{Lynne H. Walling}
\address{School of Mathematics, University of Bristol, University Walk, Clifton, Bristol BS8 1TW, United Kingdom;
phone +44 (0)117 331-5245, fax +44 (0)117 928-7978}
\email{l.walling@bristol.ac.uk}
% Abbreviated author name for running headers
%\abbrevauthor{L.H. Walling}
% Abbreviated author name for first page header
%\headabbrevauthor{Walling, L.H.}


\begin{abstract} 
We find nice representatives for the 0-dimensional cusps of the degree $n$ Siegel upper half-space under the action of $\Gamma_0(\stufe)$.  To each of these we attach a Siegel Eisenstein series, and then
we make explicit a result of Siegel, realizing any integral weight average Siegel theta series of arbitrary level $\stufe$ and Dirichlet character $\chi_L$ modulo $\stufe$ as a linear combination of Siegel Eisenstein series.
\end{abstract}

\maketitle
\def\thefootnote{}
\def\thefootnote{\arabic{footnote}}

\section{Introduction}

In the 1930's Siegel introduced generalized theta series to study quadratic forms and their representation numbers.  Given an $m\times m$ symmetric matrix $Q$ for a 
positive definite quadratic form on a $\Z$-lattice $L$, and given an $n\times n$ symmetric matrix $T$ for a positive semi-definite quadratic form, the $T$th Fourier coefficient of the degree $n$ Siegel theta series 
$\theta^{(n)}(L;\tau)$
tells us the number of dimension $n$ sublattices of $L$ on which the quadratic form $Q$ restricts to $T$.  
Siegel showed that $\theta^{(n)}(L;\tau)$ is a degree $n$, weight $m/2$ Siegel modular form of some level $\stufe$ and character $\chi_L$ modulo $\stufe$.
Further, in \cite{S35}, Siegel showed that upon averaging the theta series over the genus of $L$, the resulting average theta series $\theta^{(n)}(\gen L;\tau)$ is a linear combination of Siegel Eisenstein series, and the coefficients in this linear combination are generalized Gauss sums.

Here we consider the case when  $m=2k$ ($k\in\Z_+$) and $n<k-1$
(the condition $n<k-1$ is to ensure the convergence of the Siegel Eisenstein series we define). 
The elements of $\Gamma_{\infty}\backslash Sp_n(\Z)/\Gamma_0(\stufe)$ are sometimes called the 0-dimensional cusps of the degree $n$ Siegel upper half-space 
% $$\h_{(n)}=\{X+iY:\ X,Y\in\R^{n,n}_{\sym}\, Y>0\ \}$$
under the action of $\Gamma_0(\stufe),$
% $$\Gamma_0(\stufe)
% =\left\{\begin{pmatrix}A&B\\C&D\end{pmatrix}:\ C\equiv 0\ (\stufe)\ \right\},$$ 
and
for each $\Gamma_{\infty}\gamma\Gamma_0(\stufe)\in\Gamma_{\infty}\backslash Sp_n(\Z)/\Gamma_0(\stufe)$, there is a 
degree $n$ Siegel Eisenstein series $\E_{\gamma}$ 
transforming under $\Gamma_0(\stufe)$ with weight $k$ and character $\chi$ modulo $\stufe$
(defined in Section 3).  Varying $\gamma$ to get a complete set of representatives, we know that those $\E_{\gamma}$ that are nonzero form a basis for the space of Siegel Eisenstein series.

The majority of effort in this paper is spent on finding nice representatives for the 0-dimensional cusps.
Writing $\gamma_{_M}$ for the matrix $\begin{pmatrix}I&0\\M&I\end{pmatrix}$, 
in Section 4 we define the meaning of $\gamma_{_M}$ being a reduced representative modulo an odd prime, modulo 2, and modulo 4; we also define the meaning of $\gamma_{_M}$ being a partially reduced representative modulo $2^{e'}$ where $e'\ge 3$.
When $n=1$ or $8\nmid\stufe$, we find a complete set of representatives
$\{\gamma_{_M}\}$ for 
$\Gamma_{\infty}\backslash Sp_n(\Z)/\Gamma_0(\stufe)$
so that each $\gamma_{_M}$ is reduced modulo $\stufe$;
when $n>1$ and $8|\stufe$, we find a set $\{\gamma_{_M}\}$ that contains a complete set of representatives 
so that with $e'=\ord_2(\stufe)$,
each $\gamma_{_M}$ is reduced  modulo $\stufe/2^{e'}$ and partially reduced modulo $2^{e'}$
 (see Propositions 4.2 and 4.3).
Further, given $\gamma_{_M}$ 
so that $\gamma_{_M}$ is reduced modulo $\stufe/2^{e'}$ and partially reduced modulo $2^{e'}$,
$M$ is diagonal modulo $q^{\ord_q(\stufe)}$ for $q$ an odd prime dividing $\stufe$, and $M$ is an orthogonal sum of unary and binary blocks modulo $2^{e'}$.  Using these representatives and the local structure of the lattice $L$ at each prime dividing $\stufe$,
 it is fairly straightforward (and amusing) to evaluate the generalized Gauss sums that  give us $\theta^{(n)}(\gen L;\tau)$ as a linear combination of the Siegel Eisenstein series corresponding to these representatives $\gamma_{_M}$.
Consequently we prove the following.

\begin{thm}  Let $L$ be a rank $2k$ $\Z$-lattice ($k\in\Z_+$), and let $Q$ be a $2k\times 2k$ integral symmetric matrix defining a positive definite quadratic form on $L$ so that $Q(x)\in 2\Z$ for any $x\in L$.  Let $\stufe$ be the level of $Q$, and set $e'=\ord_2(\stufe)$.
Let $\{\gamma_{_M}\}$ be  a complete set of representatives for
$\Gamma_{\infty}\backslash Sp_n(\Z)/\Gamma_0(\stufe)$ so that when $e'\le 2$, each $\gamma_{_M}$ is reduced modulo $\stufe$, and when $e'\ge 3$ each $\gamma_{_M}$ is reduced modulo $\stufe/2^{e'}$ and partially reduced modulo $2^{e'}$.  Then 
for $n\in\Z_+$ with $n<k-1$, we have
$$\theta^{(n)}(\gen L;\tau)=\kappa\sum_M a(L,M) \E_{\gamma_{_M}}$$
where $\kappa=1$ if $\stufe>2$ and $\frac{1}{2}$ otherwise, and
$a(L,M)=\prod_{q|\stufe} a_q(L,M)$ ($q$ prime) with $a_q(L,M)$ defined as follows.
For a prime $q|\stufe$ with $q^e\parallel\stufe$, we take $G\in SL_{2k}(\Z_q)$ so that
$$^tGQG\equiv J_0\perp qJ_1\perp\cdots\perp q^eJ_e\ (q^{e+2})$$
with each $J_c$ of size $r_c\times r_c$ (some $r_c\ge0$) and $J_c$ invertible modulo $q$ when $r_c>0$; we also have
$$M\equiv M_0\perp qM_1\perp\cdots \perp q^eM_e\ (q^e)$$
with each $M_j$ of size $d_j\times d_j$ (some $d_j$) and $M_j$ invertible modulo $q$ when $d_j>0$.  Then
$$a_q(L,M)=\prod_{c=1}^e\prod_{j=0}^{c-1} q^{(j-c)r_cd_j/2}
\cdot\begin{cases} 1&\text{if $2|c-j$,}\\ 
q^{-r_cd_j/2}\G_{J_c,M_j}(q)&\text{otherwise.}\end{cases}$$
For $q$ odd,
$$\G_{J_c,M_j}(q)=\left(\frac{\det J_c}{q}\right)^{d_j}\left(\frac{\det M_j}{q}\right)^{r_c}
(\G_1(q))^{r_cd_j}$$
where $\G_1(q)$ is the classical Gauss sum; for $q=2$, $\G_{J_c,M_j}(2)$ is similar (but there are several cases), and the value of this quantity
is given explicitly in Proposition 5.5.  
\end{thm}

This theorem leaves the following questions unanswered:  how do we find a basis of Siegel Eisenstein series when $n\ge k-1$, and how do we find a complete set of representatives for $\Gamma_{\infty}\backslash Sp_n(\Z)/\Gamma_0(\stufe)$ when $n>1$ and $8|\stufe$?

The author thanks Bristol's Automorphics Anonymous, Wai Kiu Chan, and Jens Funke  for fun and helpful conversations.

\bigskip
\section{Preliminaries}

Let $L=\Z x_1\oplus\cdots\oplus\Z x_m$, a $\Z$-lattice of rank $m$, and let $Q$ be an $m\times m$ symmetric matrix with integral entries.  Thus $Q$ defines a quadratic form on $L$, via the rule that for $x=a_1x_1+\cdots+a_mx_m\in L$, we have
$$Q(x)=(a_1\ \cdots\ a_m)\,Q\ ^t(a_1\ \cdots\ a_m).$$
We assume that $Q$ defines a positive definite quadratic form on $L$, meaning that for $x\in L$, $Q(x)>0$ whenever $x\in L$ with $x\not=0$.  We also assume that $Q$ is even integral, meaning that $Q\in\Z^{n,n}_{\sym}$ with even diagonal entries
(here, for a ring $R$, $R^{n,n}_{\sym}$ denotes the set of $n\times n$ symmetric matrices with entries in $R$).  Thus for any $x\in L$, we have $Q(x)\in 2\Z$.
The level of $Q$ (also called the level of $L$) is the smallest positive integer $\stufe$ so that
$\stufe Q^{-1}$ is even integral.

For $n\in\Z_+$, we define the theta series $\theta^{(n)}(L;\tau)$ with variable
$$\tau\in\h_{(n)}=\{X+iY:\ X, Y\in\R^{n,n}_{\sym},\ Y>0\ \}$$
by setting
$$\theta^{(n)}(L;\tau)=\sum_{U\in\Z^{m,n}}\e\{Q(U)\tau\},$$
where  $Y>0$ means that $Y$ represents a positive definite quadratic form,
$\e\{*\}=\exp(\pi iTr(*))$, and $Q(U)=\,^tUQU$.
As mentioned earlier, $\theta^{(n)}(L;\tau)$ is a Siegel modular form of degree $n$, weight $m/2$, level $\stufe$ and quadratic character $\chi$ modulo $\stufe$, meaning that
with
$$Sp_n(\Z)=\left\{\begin{pmatrix}A&B\\C&D\end{pmatrix}: A\,^tB=B\,^tA,
\ C\,^tD=D\,^tC,\ A\,^tD-B\,^tC=I\ \right\}$$
and
$$\Gamma_0(\stufe)=
\left\{\begin{pmatrix}A&B\\C&D\end{pmatrix}\in Sp_n(\Z):
\ \stufe|C\ \right\},$$
for any $\gamma=\begin{pmatrix}A&B\\C&D\end{pmatrix}\in\Gamma_0(\stufe),$
we have
\begin{align*}
\theta^{(n)}(L;\tau)|\gamma
&:=\det(C\tau+D)^{-m/2}\,\theta^{(n)}(L;(A\tau+B)(C\tau+D)^{-1})\\
&=\chi_L(\det D)\,\theta^{(n)}(L;\tau).
\end{align*}
When $m$ is odd, we need to specify how we are taking square-roots; from hereon, we will assume that $m=2k$ with $k\in\Z$.  With this assumption, for $d\in\Z$ with $(d,\stufe)=1$, we have
$$\chi_L(d)=\left(\frac{(-1)^k\det Q}{|d|}\right)\,\sgn(d)^k.$$

Suppose that $L'$ is a rank $2k$ $\Z$-lattice with a positive definite quadratic form given by $Q'\in\Z^{n,n}_{\sym}$ (relative to some $\Z$-basis for $L'$).
With $L$ as above, we say that $L'$  is in the genus of $L$ if, for every prime $q$, there is some $G\in GL_{2k}(\Z_q)$ so that $^tGQ'G=Q$; here $\Z_q$ denotes the set of $q$-adic integers.
We say that $L'$ is in the same isometry class as $L$ if there is some $G\in GL_{2k}(\Z)$ so that
$^tGQ'G=Q$.
We define $o(L')$ to be the order of the orthogonal group of $L'$ (being all $G\in GL_{2k}(\Z)$ so that $^tGQ'G=Q'$), and we  set
$$\theta^{(n)}(\gen L;\tau)
=\frac{1}{\mass L}\sum_{\cls L'\in\gen L}\frac{1}{o(L')}\theta^{(n)}(L';\tau)$$
where
$$\mass L=\sum_{\cls L'\in\gen L}\frac{1}{o(L')}$$
(so the 0th Fourier coefficient of $\theta^{(n)}(\gen L;\tau)$ is 1, as is the 0th Fourier coefficient of $\theta^{(n)}(L;\tau)$).

Besides the subgroup $\Gamma_0(\stufe)$ of $Sp_n(\Z)$, we also define the subgroups
$$\Gamma(\stufe)=\{\gamma\in Sp_n(\Z):\ \gamma\equiv I\ (\stufe)\ \},$$
$$\Gamma_{\infty}=\left\{\begin{pmatrix}A&B\\0&D\end{pmatrix}\in Sp_n(\Z)\right\},$$
and
$$\Gamma_{\infty}^+=
\left\{\begin{pmatrix}A&B\\0&D\end{pmatrix}\in Sp_n(\Z):\ \det D=1\ \right\}.$$
For later convenience, we set
$G_{\pm}=\begin{pmatrix}I_{n-1}\\&-1\end{pmatrix}$ and 
$\gamma_{\pm}=\begin{pmatrix}G_{\pm}\\&G_{\pm}\end{pmatrix}.$

We repeatedly use that $Tr(AB)=Tr(BA)$ and hence $\e\{AB\}=\e\{BA\}$.  Also, with
$A,B$ denoting square matrices, we write $A\perp B$ to denote the block diagonal matrix $\diag\{A,B\}$, and for ring elements $a_1,\ldots,a_r$, we write
$\big<a_1,\ldots,a_r\big>$ to denote $\diag\{a_1,\ldots,a_r\}$.

\bigskip
\section{Siegel Eisenstein series}

In \cite{int wt}, we constructed Siegel Eisenstein series of degree $n$, weight $k\in\Z_+$, level $\stufe$ and character $\chi$ modulo $\stufe$, presuming we have $k>n+1$ (this constraint is for reasons of convergence).  Here we review this construction, making a few minor modifications to this construction, resulting in a slight modification to their normalizations; then we evaluate the Eisenstein series at the cusps.

We first define an Eisenstein series for $\Gamma(\stufe)$.
With $\delta^*$ chosen so that 
$\Gamma^+_{\infty}\Gamma(\stufe)=\cup_{\delta^*}\Gamma^+_{\infty}\delta^*
\text{ (disjoint)}$ and $\tau\in\h_{(n)}$,
we set 
$$\E^*(\tau)=\sum_{\delta^*}1(\tau)|\delta^*
\text{ where }1(\tau)|\begin{pmatrix}A&B\\C&D\end{pmatrix}=\det(C\tau+D)^{-k}.$$
Since $1(\tau)|\beta=1$ for $\beta\in\Gamma^+_{\infty}$, $\E^*$ is well-defined; further, it is analytic (in all variables of $\tau$).
Note that  for $\stufe\le 2$, we have $\gamma_{\pm}\in\Gamma(\stufe)\smallsetminus \Gamma_{\infty}^+$ and so $E^*=0$ unless $k$ is even.
% (Note that when $L$ is a lattice so that $\theta^{(n)}(L;\tau)$ has level 1 or 2, it is 
% known that the rank of $L$ must be divisible by 4 and so the weight of 
% $\theta^{(n)}(L;\tau)$ is even.)

Now take $\gamma\in Sp_n(\Z)$.  Set
$$\Gamma_{\gamma}=\{\beta\in\Gamma_0(\stufe):\ \Gamma_{\infty}\Gamma(\stufe)\gamma\beta=\Gamma_{\infty}\Gamma(\stufe)\gamma\ \},$$
and 
$$\Gamma^+_{\gamma}=\{\beta\in\Gamma_0(\stufe):\ \Gamma^+_{\infty}\Gamma(\stufe)\gamma\beta=\Gamma^+_{\infty}\Gamma(\stufe)\gamma\ \};$$
one easily checks that $[\Gamma_{\gamma}:\Gamma^+_{\gamma}]=1$ or 2.
Choose $\delta, \delta'$ so that 
$$\Gamma_0(\stufe)=\cup_{\delta}\Gamma^+_{\gamma}\delta \text{ (disjoint), }
\Gamma^+_{\gamma}=\cup_{\delta'}\Gamma(\stufe)\delta' \text{ (disjoint);}$$
using that $\Gamma(\stufe)$ is a normal subgroup of $Sp_n(\Z)$, we see that
$$\Gamma_{\infty}^+\gamma\Gamma_0(\stufe)
=\cup_{\delta',\delta}\Gamma_{\infty}^+\Gamma(\stufe)\gamma\delta'\delta.$$
We set
\begin{align*}
\E'_{\gamma}
&=\sum_{\delta,\delta'}\overline\chi(\delta\delta')\E^*|\gamma\delta'\delta;
\end{align*}
here $\chi(\delta)$ means $\chi(\det D_{\delta})$.
Note that for $\beta\in\Gamma^+_{\gamma}$, we have
$$\cup_{\delta^*}\Gamma^+_{\infty}\delta^*\gamma\beta
=\Gamma^+_{\infty}\Gamma(\stufe)\gamma\beta
=\Gamma^+_{\infty}\Gamma(\stufe)\gamma
=\cup_{\delta^*}\Gamma^+_{\infty}\delta^*\gamma,$$
and so $\E^*|\gamma\beta=\E^*|\gamma$;
hence $\E'_{\gamma}$ is well-defined.  Also,
for any $\alpha\in \Gamma_0(\stufe)$, $\delta\alpha$ varies over a set of coset representatives for $\Gamma_{\gamma}^+\backslash\Gamma_0(\stufe)$ as $\delta$ does, and so 
$\E'_{\gamma}|\alpha=\chi(\alpha)\E'_{\gamma}$.  Notice that
$$\E'_{\gamma}=
\left(\sum_{\delta'}\overline\chi(\delta')\right)
\sum_{\delta}\overline\chi(\delta)\E^*|\gamma\delta,$$
and thus $\E'_{\gamma}=0$ unless $\chi$ is trivial on $\Gamma^+_{\gamma}$.
Also notice that with $\gamma,\delta',\delta$ as above, we have
$$\Gamma_{\infty}\gamma\Gamma_0(\stufe)
=\cup_{\delta',\delta}\big(
\Gamma_{\infty}^+\Gamma(\stufe)\gamma\delta'\delta 
\cup \Gamma_{\infty}^+\Gamma(\stufe)\gamma_{\pm}\gamma\delta'\delta\big)$$
and $\E'_{\gamma_{\pm}\gamma}=(-1)^k\E'_{\gamma}$.

For $\gamma\in Sp_n(\Z)$, set
$$\E_{\gamma}=\frac{1}{[\Gamma_{\gamma}:\Gamma(\stufe)]}\E'_{\gamma}.$$
So when $\E_{\gamma}\not=0$ and
$\Gamma^+_{\gamma}=\Gamma_{\gamma}$, we have
$$\E_{\gamma}=
\sum_{\delta\in\Gamma_{\gamma}\backslash\Gamma_0(\stufe)}
\overline\chi(\delta)\E^*|\gamma\delta.$$
Now suppose that $\E_{\gamma}\not=0$ and  $\Gamma^+_{\gamma}\not=\Gamma_{\gamma}$; take $\beta'\in\Gamma_{\gamma}\smallsetminus\Gamma^+_{\infty}.$
Then
\begin{align*}
\E_{\gamma}
&=\frac{1}{2}
\sum_{\delta\in\Gamma_{\gamma}\backslash\Gamma_0(\stufe)}\overline\chi(\delta)
\E^*|\gamma\delta
+\frac{1}{2}
\sum_{\delta\in\Gamma_{\gamma}\backslash\Gamma_0(\stufe)}\overline\chi(\beta'\delta)
\E^*|\gamma\beta'\delta.
\end{align*}
By our choice of $\beta'$, we have $\gamma\beta'\gamma^{-1}\in\gamma_{\pm}\Gamma^+_{\infty}\Gamma(\stufe),$
so $\E^*|\gamma\beta'=\chi(-1)\E^*|\gamma.$
Hence
\begin{align*}
\E_{\gamma}=\frac{1}{2}\big(1+\overline\chi(\beta')\chi(-1)\big)
\sum_{\delta\in\Gamma_{\gamma}\backslash\Gamma_0(\stufe)}
\overline\chi(\delta)\E^*|\gamma\delta,
\end{align*}
and so $\E_{\gamma}=0$ unless $\chi(\beta')=\chi(-1)$.

Thus regardless of whether $\Gamma_{\gamma}^+=\Gamma_{\gamma}$, when $\E_{\gamma}\not=0$  we have
$$\E_{\gamma}=\sum_{\delta\in\Gamma_{\gamma}\backslash\Gamma_0(\stufe)}
\overline\chi(\delta)\E^*|\gamma\delta.$$

As discussed in \cite{int wt}, as $\Gamma_{\infty}\gamma\Gamma_0(\stufe)$ varies over $\Gamma_{\infty}\backslash Sp_n(\Z)/\Gamma_0(\stufe)$, the non-zero $\E_{\gamma}$ form a basis for the space of Siegel Eisenstein series of degree $n$, weight $k$, level $\stufe$, and character $\chi$.

Now we evaluate the non-zero $\E_{\gamma}$ at the cusps.

\begin{prop}  Suppose that $\alpha,\gamma\in Sp_n(\Z)$ so that $\E_{\gamma}\not=0$.
If $\alpha\not\in\Gamma_{\infty}\gamma\Gamma_0(\stufe)$, then
$$\lim_{\tau\to i\infty}\E_{\gamma}(\tau)|\alpha^{-1}=0.$$
If $\alpha=\beta\gamma\delta'$ for some $\beta\in\Gamma_{\infty}$ and $\delta'\in\Gamma_0(\stufe)$, then
\begin{align*}
\lim_{\tau\to i\infty}\E_{\gamma}(\tau)|\alpha^{-1}
&=\begin{cases}
2&\text{if $\stufe\le 2$,}\\
\overline\chi(\delta'\beta)&\text{otherwise.}\end{cases}
\end{align*}
\end{prop}

\begin{proof}  Since $\E_{\gamma}\not=0$, we have $\E^*\not=0$ (so if $\stufe\le 2$, $k$ must be even).
In \cite{int wt}, we saw that
$$\lim_{\tau\to i\infty}\E^*(\tau)=
\begin{cases}2&\text{if $\stufe\le 2$,}\\1&\text{if $\stufe>2$.}\end{cases}$$
Thus $\lim_{\tau\to i\infty}\E_{\gamma}|\alpha^{-1}=0$ unless there is some $\delta\in\Gamma_0(\stufe)$ so that $\gamma\delta\alpha^{-1}\in\Gamma_{\infty}\Gamma(\stufe)$; so this limit is 0 whenever $\alpha\not\in\Gamma_{\infty}\gamma\Gamma_0(\stufe)$.

Now  suppose that  $\alpha=\beta\gamma\delta'$ for some $\beta\in\Gamma_{\infty}$ and some $\delta'\in\Gamma_0(\stufe)$.
Thus $$\E_{\gamma}|\alpha^{-1}=\overline\chi(\delta')\E_{\gamma}|\gamma^{-1}\beta^{-1}.$$
Also,
\begin{align*}
\lim_{\tau\to i\infty}\E_{\gamma}(\tau)|\gamma^{-1}\beta^{-1}
&=\lim_{\tau\to i\infty}\sum_{\delta\in\Gamma_{\gamma}\backslash\Gamma_0(\stufe)}
\overline\chi(\delta)\E^*(\tau)|\gamma\delta\gamma^{-1}\beta^{-1}\\
&=\lim_{\tau\to i\infty}\overline\chi(\beta)\E^*(\tau)
\end{align*}
as $\gamma\delta\gamma^{-1}\beta^{-1}\in\Gamma_{\infty}\Gamma(\stufe)$ if and only if $\delta\in\Gamma_{\gamma}$ (in which case $\overline\chi(\delta)\E^*|\gamma\delta\gamma^{-1}=\E^*$).
Hence 
$$\lim_{\tau\to i\infty}\E_{\gamma}(\tau)|\alpha^{-1}=\overline\chi(\delta'\beta)
\lim_{\tau\to i\infty}\E^*(\tau).$$
(Note that $\chi$ is trivial when $\stufe\le 2$.)
\end{proof}

\smallskip\noindent{\bf Remark.}  Suppose that $\E_{\gamma}\not=0$.  Recall that earlier we noticed that $\E_{\gamma_{\pm}\gamma}=(-1)^k\E_{\gamma}.$  Thus with $\kappa=1/2$ when $\stufe\le 2$ and $\kappa=1$ otherwise, by the above proposition we have
\begin{align*}
\kappa\overline\chi(\gamma_{\pm})
&=\lim_{\tau\to i\infty}\E_{\gamma}(\tau)|(\gamma_{\pm}\gamma)^{-1}\\
&=(-1)^k\lim_{\tau\to i\infty}\E_{\gamma_{\pm}\gamma}(\tau)|(\gamma_{\pm}\gamma)^{-1}\\
&=\kappa(-1)^k.
\end{align*}
Hence when $\E_{\gamma}\not=0$, we have $\chi(-1)=(-1)^k$.

\bigskip
\section{Representatives for 0-dimensional cusps}

In this section, we assume that $\stufe$ is odd and
we determine a set of representatives for 
the 0-dimensional cusps, each of which corresponds to an element of
$\Gamma_{\infty}\backslash Sp_n(\Z)/\Gamma_0(\stufe)$.
The representatives we find are of the form $\begin{pmatrix}I&0\\M&I\end{pmatrix}.$

\smallskip
\noindent{\bf Definition.}  
Take $M\in\Z^{n,n}_{\sym}$ and set $\gamma_{_M}=\begin{pmatrix}I&0\\M&I\end{pmatrix}$
(so $\gamma_{_M}\in Sp_n(\Z)$).
Set $H=\begin{pmatrix}0&1\\1&0\end{pmatrix}$ and $A=\begin{pmatrix}2&1\\1&2\end{pmatrix}$.  We write $H^d$ to denote the orthogonal sum of $d$ copies of $H$.  
\begin{enumerate}
\item[(a)]   Let $q$ be an odd prime; fix $\omega$ so that $\left(\frac{\omega}{q}\right)=-1$.  For $e\in\Z_+$, we say that $\gamma_{_M}$ is a reduced representative modulo $q^e$ if the following conditions are met.
\begin{enumerate}
\item[(i)]  $M\equiv M_0\perp qM_1\perp\cdots\perp q^eM_e\ (q^e)$ with each $M_j$ $d_j\times d_j$ and invertible modulo $q$; take $\ell$ minimal so that $d_{\ell}>0$ and take $h$ maximal so that $d_h>0$;
\item[(ii)]  if $\ell<j<e$
with $d_j>0$ then  $M_j=\big<1,\ldots,1,\varepsilon_j\big>$
where $\varepsilon_j=1$ or $\omega$;
\item[(iii)]  if $0<\ell<h=e$ then $M_{\ell}=\big<1,\ldots,1,\varepsilon_{\ell}\big>$
where $\varepsilon_{\ell}=1$ or $\omega$;
\item[(iv)]  if $\ell\le h<e$ then $M_{\ell}=\big<1,\ldots,1,\varepsilon_{\ell}\big>$ where $1\le \varepsilon\le q^{\min(\ell,e-h)}$, $q\nmid\varepsilon_{\ell}$.
\end{enumerate}
\item[(b)]  For $n=1$ and $e\in\Z_+$, we say that $\gamma_{_M}$ is a reduced representative modulo $2^e$ if $M\equiv 2^{\ell}\varepsilon\ (2^e)$ where $1\le \varepsilon\le 2^{\min(\ell,e-\ell)}$ with $2\nmid\varepsilon$.
\item[(c)]  For $n>1$, we say that $\gamma_{_M}$ is a reduced representative modulo $2$ if for some $d\in\Z$, $M\equiv I_d\perp 0_{n-d}\ (2)$; we say that $\gamma_{_M}$ is a reduced representative modulo 4 if for some $d\in\Z$,
$M\equiv I_d\perp 2J_1\perp 4J_2\ (4)$ where either $J_1= I_{d'}$
or $J_1=H\perp\cdots\perp H$.
\item[(d)]  For $n>1$ and $e\ge 3$, we say that $\gamma_{_M}$ is a partially reduced representative modulo $2^e$ if the following conditions are met.
\begin{enumerate}
\item[(i)]  $M\equiv M_0\perp 2M_0\perp\cdots\perp 2^eM_e\ (2^e)$ where each $M_j$ is $d_j\times d_j$ and invertible modulo 2;  take $\ell$ minimal so that $d_{\ell}>0$;
\item[(ii)] if $d_0>0$ then $M_0\equiv I\ (2^e)$;
\item[(iii)]   if $\ell<j<e$ with $d_j>0$, then either $M_j$ is diagonal with diagonal entries from the set $\{1,3,5,7\}$, or $M_j=H^{d_j/2}$, or
$M_j=H^{d_j/2-1}\perp A$;
\item[(iv)]  if $0<\ell<e$ then either $M_{\ell}=\big<\eta_1,\ldots,\eta_{d_{\ell}}\big>$ with
$\eta_1,\ldots,\eta_{d_{\ell}-1}\in\{1,3,5,7\}$ and $\eta_{d_{\ell}}$ odd, or
$M_{\ell}=H^{d_{\ell}/2-1}\perp A'$ where $A'=\begin{pmatrix}2a'&a\\a&2a'a^2\end{pmatrix}$ with $a$ odd and $a'=0$ or 1.
\end{enumerate}
\end{enumerate}
For $\stufe\in\Z_+$ with $8\nmid\stufe$, we say that $\gamma_{_M}$ is a reduced representative modulo $\stufe$ if $\gamma_{_M}$ is a reduced representative modulo $q^e$ for each prime $q|\stufe$ with $q^e\parallel\stufe$.  
\smallskip
\smallskip

We will show that each element of $\Gamma_{\infty}\backslash Sp_n(\Z)/\Gamma_0(\stufe)$ is represented by exactly one reduced representative modulo $\stufe$.  We begin with the following easy proposition.

\begin{prop}  Fix $\stufe\in\Z_+$.
\begin{enumerate}
\item[(a)]  Suppose that $\delta\in Sp_n(\Z)$.  Then there is some $M''\in\Z^{n,n}_{\sym}$
% $\gamma_0=\begin{pmatrix}I&0\\M''&I\end{pmatrix}\in Sp_n(\Z)$ 
so that
$\delta\in\Gamma_{\infty}\gamma_{_{M''}}\Gamma_0(\stufe).$
\item[(b)]  Suppose that $M,M''\in\Z^{n,n}_{\sym}$ so that $G(M''\ I)\beta\equiv (M\ I)\ (\stufe)$ where $G\in GL_n(\Z)$ and $\beta\in \Gamma_0(\stufe)$.
Then
$\gamma_{_{M''}}\in\Gamma_{\infty}
\gamma_{_{M}}\Gamma_0(\stufe).$
\end{enumerate}
\end{prop}

\begin{proof}  
(a)  Write $\delta=\begin{pmatrix}A&B\\C&D\end{pmatrix}$.  By Proposition 3.4 \cite{int wt},
there is some $M''\in\Z^{n,n}_{\sym}$ so that 
$(C\ D)=E(M''\ I)\gamma$
for some $E\in SL_n(\Z)$ and $\gamma\in\Gamma_0(\stufe).$
Hence with $$\beta=\begin{pmatrix}^tE^{-1}\\&E\end{pmatrix}
\text{ and }\gamma_0=\begin{pmatrix}I&0\\M''&I\end{pmatrix},$$
we have $\beta\in\Gamma_{\infty}$ and
$$\beta\gamma_0\gamma
=\begin{pmatrix}*&*\\C&D\end{pmatrix}\in Sp_n(\Z).$$
Therefore
$\delta\in\Gamma_{\infty}\beta\gamma_0\gamma
\subseteq \Gamma_{\infty}\gamma_0\Gamma_0(\stufe).$

(b)  By Proposition 3.3 \cite{int wt}, we have $G(M''\ I)\beta\in(M\ I)\Gamma(\stufe)$
and hence
$$\begin{pmatrix}^tG^{-1}\\&G\end{pmatrix}\begin{pmatrix}I&0\\ M''&I\end{pmatrix}\beta
\in
\Gamma_{\infty}\begin{pmatrix}I&0\\M&I\end{pmatrix}\Gamma(\stufe).$$
From this the claim easily follows.
\end{proof}

\begin{prop}  Let $\stufe\in\Z_+$, and take $\delta\in Sp_n(\Z)$.  Set $e'=\ord_2(\stufe)$.  When $e'\le 2$ or $n=1$ there is a reduced representative $\gamma_{_M}$ modulo $\stufe$ so that  
$\delta\in\Gamma_{\infty}\gamma_{_M}\Gamma_0(\stufe)$.  When $e'\ge 3$, there is some $\gamma_{_M}\in Sp_n(\Z)$ with 
$\gamma_{_M}$ reduced modulo $\stufe/2^{e'}$, $\gamma_{_M}$ partially reduced modulo $2^{e'}$, and 
$\delta\in\Gamma_{\infty}\gamma_{_M}\Gamma_0(\stufe)$.
\end{prop}

\begin{proof}  
By Proposition 4.1, there is some $M''\in\Z^{n,n}_{\sym}$ so that
$\delta\in\Gamma_{\infty}\gamma_{_{M''}}\Gamma_0(\stufe).$
We show that
$\gamma_{_{M''}}\in\Gamma_{\infty}
\gamma_{_{M}}\Gamma_0(\stufe)$
where $\gamma_{_{M}}$ is a reduced representative modulo $\stufe$, and hence
$\delta\in\Gamma_{\infty}\gamma_{_{M}}\Gamma_0(\stufe).$

To do this,
for each prime $q|\stufe$ with $q^e\parallel\stufe$, 
we find matrices $E(q), G(q)\in SL_n(\Z)$ with $E(q)G(q)\equiv I\ (\stufe/q^e)$, and 
$\alpha(q)\beta(q)\in\Gamma_0(\stufe)$ with $\alpha(q)\beta(q)\equiv I\ (\stufe/q^e)$, and so that
$$^tE(q)\,^tG(q)(M''\ I)\alpha(q)\beta(q)\equiv (M\ I)\ (q^e)$$
where $\gamma_{_M}$ is reduced modulo $q^e$ (or partially reduced when $q=2$ and $e\ge 3$).
Then we define $E(\stufe)$, $G(\stufe)$, $\alpha(\stufe)$, $\beta(\stufe)$ by setting $E(\stufe)=\prod_{q|\stufe}E(q)$ and so on. Thus we get 
$$^tE(\stufe)\,^tG(\stufe)(M''\ I)\alpha(\stufe)\beta(\stufe)\equiv (M\ I)\ (\stufe).$$
Consequently, by Proposition 4.1, 
% $\gamma_{_{M''}}\in\Gamma_{\infty}\gamma_{_M}\Gamma_0(\stufe),$
% and hence 
$\delta\in \Gamma_{\infty}\gamma_{_M}\Gamma_0(\stufe).$

We first consider the case that $q$ is odd.

(a) Fix an odd prime $q|\stufe$ and $e\in\Z_+$ so that $q^e\parallel\stufe$; fix $\omega\in\Z$ so that $\left(\frac{\omega}{q}\right)=-1$. 
We know by \S91 \cite{O'M}, or equivalently Corollary 8.2 and Theorem 85 of \cite{Ger}, that
 there is some $G''\in GL_n(\Z_q)$ so that
$$^tG''M''G''=M_0''\perp qM_1''\perp\cdots\perp q^eM''_e$$
with $Mji''$ of size $d_j\times d_j$ for some $d_j$, and when $d_j>0$ with $j<e$,
$M_j''=\big<1,\ldots,1,\eta''_j\big>$ where $\eta''_j=1$ or $\omega$.
Fix $\ell$ to be minimal with $d_{\ell}>0$.  Then
right-multiplying $G''$ by a suitable diagonal matrix, we obtain $G'\in SL_n(\Z_q)$ so that
$$^tG'M''G'=q^{\ell}M'_{\ell}\perp q^{\ell+1}M'_{\ell+1}\perp
\cdots\perp q^eM'_e$$
where 
$M'_j=\big<1,\ldots,1,\eta'_j\big>$, with  $\eta_j'=\eta_j''$ for $\ell<j<e$ when $d_j>0$,
and $\eta'_{\ell}=\eta''_{\ell}(\det G'')^{-2}.$  
Now take $G=G(q)\in SL_n(\Z)$ so that $G\equiv I\ (\stufe/q^e)$ and
$G\equiv G'\ (q^e)$, and set $M'=\,^tGM''G$.
Thus $M'\equiv q^{\ell}M'_{\ell}\perp \cdots\perp q^eM'_e\ (q^e).$
Set $\alpha=\alpha(q)=\begin{pmatrix}G\\&^tG^{-1}\end{pmatrix}$; then $\alpha\in\Gamma_0(\stufe)$ with $\alpha\equiv I\ (\stufe/q^e)$,
% with $\alpha\equiv I\ (\stufe/q^e)$ 
and
$$^tG(M''\ I)\alpha= (\,^tGM''G\ I).$$

We now find $E=E(q)\in SL_n(\Z)$ and $\beta=\beta(q)\in\Gamma_0(\stufe)$ so that 
$E(M'\ I)\beta\equiv(M'\ I)\ (\stufe/q^e)$ and  
$E(M'\ I)\beta\equiv(M\ I)\ (q^e)$ where $\gamma_{_M}$ is a reduced representative modulo $q^e$.

First note that for $\ell<j<e$ we have $M'_j=\big<1,\ldots,1,\eta'_j\big>=\varepsilon_j$ where $\eta'_j=1$ or $\omega$.  

Suppose that $0<\ell< h=e$.  Take $u\in\Z$ so that 
$\eta'_{\ell}u^2\equiv 1$ or $\omega$ modulo $q^e$, and take $\overline u$ so that 
$u\overline u\equiv 1\ (q^e)$.
Take $E'=\begin{pmatrix}w&x\\y&z\end{pmatrix}\in SL_2(\Z)$ so that $E'\equiv \begin{pmatrix}u\\&\overline u\end{pmatrix}\ (q^e)$ and $E'\equiv I\ (\stufe/q^e)$.
Take $\beta'=\begin{pmatrix}E'\\&^t(E')^{-1}\end{pmatrix}.$
Thus 
$$^tE'\begin{pmatrix}q^{\ell}\eta'_{\ell}&&1\\&0&&1\end{pmatrix}\beta'
\equiv\begin{pmatrix}q^{\ell}\eta'_{\ell}u^2&&1\\&0&&1\end{pmatrix}\ (q^e).$$
We lift $E'$ to
$$E=E(q)=\begin{pmatrix}W&X\\Y&Z\end{pmatrix}\in SL_n(\Z)$$ with $E\equiv I\ (\stufe/q^e)$ by taking
$$W=I_{d_{\ell}-1}\perp\big<w\big>,
\ X=0_{d_{\ell}-1}\perp\big<x\big>,$$
$$Y=0_{n-d_{\ell}-1}\perp\big<y\big>,
\ Z=I_{n-d_{\ell}-1}\perp\big<z\big>.$$
Set $\beta=\beta(q)=\begin{pmatrix}E\\&^tE^{-1}\end{pmatrix}.$
Then $$^tE(M'\ I)\beta\equiv(M\ I)\ (q^e)$$
where $\gamma_{_M}$ is reduced modulo $q^e$.

Now suppose that $\ell\le h<e$ and $\ell\le e-h$.  Choose $\eta_{\ell}$ so that $1\le\eta_{\ell}\le q^{\ell}$ with $\eta_{\ell}\equiv\eta'_{\ell}\ (q^{\ell}).$  Thus with 
$\overline \eta_{\ell}\in\Z$ so that $\overline \eta_{\ell}\eta_{\ell}\equiv 1\ (q^e)$, we have $\overline\eta_{\ell}\eta'_{\ell}=1+q^{\ell} b'$ for some $b'\in\Z$.  Take  $b=-\overline\eta'_{\ell}b'$, and take $\beta'=\begin{pmatrix}w&x\\y&z\end{pmatrix}\in SL_2(\Z)$ so that $\beta'\equiv I\ (\stufe/q^e)$ and 
$\beta'\equiv
\begin{pmatrix}\overline\eta'_{\ell}\eta_{\ell}&b\\
&\eta'_{\ell}\overline\eta_{\ell}\end{pmatrix}\ (q^e).$
Then
$$(q^{\ell}\eta'_{\ell}\ \,1)\beta'\equiv(q^{\ell}\eta_{\ell}\ \,1)\ (q^e).$$
We lift $\beta'$ to $\beta=\beta(q)\in\Gamma_0(\stufe)$ with $\beta\equiv I\ (\stufe/q^e)$ by setting $\beta=\begin{pmatrix}W&X\\Y&Z\end{pmatrix}$ where
$$W=I_{d_{\ell}-1}\perp\big<w\big>\perp I_{n-\ell},
\ X=0_{d_{\ell}-1}\perp\big<x\big>\perp 0_{n-\ell},$$
$$Y=0_{d_{\ell}-1}\perp\big<y\big>\perp 0_{n-\ell},
\ Z=I_{d_{\ell}-1}\perp\big<z\big>\perp I_{n-\ell}.$$
Then $(M'\ I)\beta\equiv (M\ I)\ (q^e)$ where $\gamma_{_M}$ is reduced modulo $q^e$.
We set $E(q)=I$.

Finally, suppose that $\ell\le h<e$ and
$0<e-h<\ell<e$.   Choose $\eta_{\ell}$ so that $1\le \eta_{\ell}\le q^{e-h}$ with $\eta_{\ell}\equiv\eta'_{\ell}\ (q^{e-h})$.  As $e-h>0$, we know that 
$\left(\frac{\eta_{\ell}\eta'_{\ell}}{q}\right)=1$ and so there is some $u\in\Z$ so that $\eta'_{\ell}u^2\equiv\eta_{\ell}\ (q^e)$.  
Take $E'=\begin{pmatrix}w&x\\y&z\end{pmatrix}\in SL_2(\Z)$ so that $E'\equiv I\ (\stufe/q^e)$
and $E'\equiv\begin{pmatrix}u\\&\overline u\end{pmatrix}\ (q^e)$ where $\overline uu\equiv 1\ (q^e).$  Take $\beta'=\begin{pmatrix}E'\\&^t(E')^{-1}\end{pmatrix}.$
Then
$$^tE'
\begin{pmatrix}q^{\ell}\eta'_{\ell}&&1\\&q^h\eta'_{h}&&1\end{pmatrix}
\beta'\equiv\begin{pmatrix}q^{\ell}\eta_{\ell}&&1
\\&q^h\eta'_{h}\overline u^2&&1\end{pmatrix}\ (q^e).$$
We have $\overline u^2\equiv\eta'_{\ell}\overline\eta_{\ell}\equiv1\ (q^{e-h})$, and thus
$q^h\eta'_h\overline u^2
\equiv q^h\eta'_h\ (q^e).$
We lift $E'$ to $E=E(q)\in SL_n(\Z)$ with $E\equiv I\ (\stufe/q^e)$
by setting $E'=\begin{pmatrix}W&X\\Y&Z\end{pmatrix}$ where 
$$W=I_{d_{\ell}-1}\perp\big<w\big>,
\ X=0_{d_{\ell}-1}\perp\big<x\big>,$$
$$Y=0_{n-d_{\ell}-1}\perp\big<y\big>,
\ Z=I_{n-d_{\ell}-1}\perp\big<z\big>.$$
Set $\beta=\beta(q)=\begin{pmatrix}E\\&^tE^{-1}\end{pmatrix}.$
So $\beta\in\Gamma_0(\stufe)$ with $\beta\equiv I\ (\stufe/q^e)$, and
$$^tE(M'\ I)\beta\equiv(M\ I)\ (q^e)$$
where $\gamma_{_M}$ is reduced modulo $q^e$.

(b)  Now suppose that $n=1$, and fix a prime $q|\stufe$ with $q^e\parallel\stufe$; take $\ell$ and $\eta'$ so that $M''=q^{\ell}\eta'$, $q\nmid\eta'$.  If $\ell\ge e$ then $\gamma_{_{M''}}$ is a reduced representative modulo $q^e$.  So suppose that $\ell<e$.

Suppose that $e-\ell\le \ell$.  Take $\eta$ so that $1\le \eta\le q^{e-\ell}$ with $\eta\equiv\eta'\ (q^{e-\ell})$. Then $(q^{\ell}\eta'\ \,1)\equiv(q^{\ell}\eta\ \,1)\ (q^e)$, and so $\gamma_{_{M''}}$ is reduced modulo $q^e$.  Take $E(q), G(q), \alpha(q), \beta(q)$ to be identity matrices.

Suppose that $\ell<e-\ell$.  Choose $\eta$ so that $1\le\eta\le q^{\ell}$ with $\eta\equiv\eta'\ (q^{\ell})$.  Take $u$ so that $u\equiv\eta\overline{\eta'}\ (q^e)$ where $\overline{\eta'}\eta'\equiv1\ (q^e)$.  Thus with $\overline u$ so that $\overline u u\equiv 1\ (q^e)$, we have $\overline u=1+q^{\ell}b'$ for some $b'$.  Take $b= -\overline{\eta'}b'$, and take $\beta=\beta(q)\in SL_2(\Z)$ so that $\beta\equiv I\ (\stufe/q^e)$ and $\beta\equiv\begin{pmatrix}u&b\\0&\overline u\end{pmatrix}\ (q^e).$
Thus $(M''\ 1)\beta=(M\ 1)$ where $\gamma_{_M}$ is reduced modulo $q^e$.
Take $E(q), G(q), \alpha(q)$ to be identity matrices.

\smallskip
(c)  Suppose that $n>1$ and $2^e\parallel\stufe$ where $e>0$.
By  \S93 of \cite{O'M}, or equivalently Theorem 8.9 of \cite{Ger}, there is some $G''\in GL_n(\Z_2)$ so that
$$^tG''M''G''=M_0''\perp 2M''_1\perp\cdots\perp 2^eM''_e$$
with $M_j''$ $d_j\times d_j$ for some $d_j$, and when $d_j>0$ with $j<e$, either $M_j''$
is diagonal with entries from $\{1,3,5,7\}$, or $M_j''=H\perp\cdots\perp H$
where $H=\begin{pmatrix}0&1\\1&0\end{pmatrix}$, or $M_j''=H\perp\cdots\perp H\perp A$
where $A=\begin{pmatrix}2&1\\1&2\end{pmatrix}.$
Then as in case (a), we can take $G=G(2)\in SL_n(\Z)$ so that $G\equiv I\ (\stufe/2^e)$ and
$$M'=\,^tGM''G\equiv 2^{\ell}M'_{\ell}\perp\cdots\perp 2^eM'_e\ (2^e)$$
where  $M_j'=M_j''$ for $\ell<j< e$, and
$$M'_{\ell}=\begin{pmatrix}I\\&(\det G'')^{-1}\end{pmatrix}M_{\ell}''
\begin{pmatrix}I\\&(\det G'')^{-1}\end{pmatrix}.$$
Set $\alpha(2)=\begin{pmatrix}G\\&^tG^{-1}\end{pmatrix}.$

Note that if $d_1>0$ then $2M_1''\equiv 2I_{d_1}\ (4)$ or
$2M_1''\equiv H\perp\cdots\perp H\ (4);$ hence if $d_0=0$, then $\gamma_{_{M'}}$ is
completely reduced modulo $2^e$ if $e\le 2$, and partially reduced modulo $2^e$ if $e\ge 3$.
So when $d_0=0$ we take $E(2)=I$ and $\beta(2)=I$.

Suppose that $d_0>0$.   
Take $v\in\Z$ so that $v\equiv\det M_0'\ (2^e)$ and $v\equiv 1\ (\stufe/2^e)$;
take $\begin{pmatrix}w&x\\y&z\end{pmatrix}\in SL_2(\Z)$ so that
$\begin{pmatrix}w&x\\y&z\end{pmatrix}\equiv\begin{pmatrix}\overline v\\&v\end{pmatrix}\ (\stufe)$
where $\overline v v\equiv 1\ (\stufe).$
Take $\delta'=\begin{pmatrix}W&X\\Y&Z\end{pmatrix}$ where
$$W=I_{d_0}\perp\big< w\big>\perp I_{n-d_0},
\ X=0_{d_0}\perp\big< x\big>\perp 0_{n-d_0},$$
$$Y=0_{d_0}\perp\big< y\big>\perp 0_{n-d_0},
\ Z=I_{d_0}\perp\big< z\big>\perp I_{n-d_0}.$$
Now take $\overline{M_0'}$ so that $\overline{M_0'}M_0'\equiv I\ (2^e)$, and
take $U\in SL_{d_0}(\Z)$ so that $U\equiv\begin{pmatrix}I\\&v\end{pmatrix}\overline{M_0''}\ (2^e)$
and $U\equiv I\ (\stufe/2^e).$  Take $U'$ so that $U'\equiv U(I-M_0')\ (2^e)$ and $U'\equiv 0\ (\stufe/2^e).$
Set
$$\delta''=\begin{pmatrix}U&&U'\\&I_{n-d_0}\\&&^tU^{-1}\\&&&I_{n-d_0}\end{pmatrix}.$$
So $\delta''\in\Gamma_0(\stufe)$ and $\delta''\equiv I\ (\stufe/2^e)$.
Set $\beta=\beta(2)=\delta'\delta''$ and $E=E(2)=I$.
Then $^tE\,^tG(M''\ I)\alpha\beta\equiv (M\ I)\ (2^e)$ where $\gamma_{_M}$ is reduced modulo $2^e$
if $e\le 2$, and partially reduced modulo $2^e$ if $e\ge 3$.
\end{proof}

\begin{prop}  
Take $\stufe\in\Z_+$ and set $e'=\ord_2(\stufe)$.  Suppose that 
$\gamma_{_{M'}}\in\Gamma_{\infty}\gamma_{_M}\Gamma_0(\stufe)$ where
$\gamma_{_M}$ and $\gamma_{_{M'}}$ are reduced representatives modulo $\stufe/2^{e'}$; also suppose that  $\gamma_{_M}$ and $\gamma_{_{M'}}$ are reduced representatives modulo $2^{e'}$ when $e'\le 2$, and partially reduced representatives modulo $2^{e'}$ when $e'\ge 3$.
Then $M'\equiv M\ (\stufe/2^{e'})$, and $M'\equiv M\ (2^{e'})$ if $e'\le 2$.  Hence for $e'\le 2$, as $\gamma_{_M}$ varies over reduced representatives modulo $\stufe$, $\{\gamma_{_M}\}$ is a complete set of representatives for
$\Gamma_{\infty}\backslash Sp_n(\Z)/\Gamma_0(\stufe)$; for $e'\ge 3$, as $\gamma_{_M}$ varies over representatives that are reduced modulo $\stufe/2^{e'}$ and partially reduced modulo $2^{e'}$, the set $\{\gamma_{_M}\}$ contains a set of representatives for 
$\Gamma_{\infty}\backslash Sp_n(\Z)/\Gamma_0(\stufe)$.
\end{prop}

\begin{proof} 
Since $\gamma_{_{M'}}\in\Gamma_{\infty}\gamma_{_M}\Gamma_0(\stufe)$, we that there are $E\in GL_n(\Z)$ and $\delta=\begin{pmatrix}A&B\\C&D\end{pmatrix}\in\Gamma_0(\stufe)$ 
so that $E(M\ I)\delta=(M'\ I)$.  Set $(M''\ I)=G_{\pm}(M'\ I)\gamma_{\pm}$.  
Since $M'$ is diagonal modulo $\stufe/2^{e'}$, we have $M''\equiv M'\ (\stufe/2^{e'})$;
when $e'\le 2$, we have $M''\equiv M'\ (2^{e'})$.
Thus replacing $\gamma_{_{M'}}$ by $\gamma_{\pm}\gamma_{_{M'}}\gamma_{\pm}$ if necessary, we can assume that $E(M\ I)\delta=(M'\ I)$ with $E\in SL_n(\Z)$.

For $e'\le 2$, Proposition 3.4 \cite{int wt} shows that $M'\equiv M\ (2^{e'})$.

So suppose that $q$ is an odd prime dividing $\stufe$ with $q^e\parallel\stufe$.
We have 
$$M\equiv M_0\perp qM_1\perp\cdots\perp q^eM_e\ (q^e)$$
with $M_i$ $d_i\times d_i$ for some $d_i$, and when $d_i>0$,
$M_i=\big<1,\ldots,1,\varepsilon_i\big>$ where $\varepsilon_i$ is as in the definition of a reduced representative modulo $q^e$.  Similarly,
$$M'\equiv M'_0\perp qM'_1\perp\cdots\perp q^eM'_e\ (q^e)$$
with $M_i$ $d_i\times d_i$ for some $d_i$, and when $d_i>0$,
$M'_i=\big<1,\ldots,1,\varepsilon'_i\big>$ where $\varepsilon'_i$ is as in the definition of a reduced representative modulo $q^e$.
Take $\ell$ minimal so that $d_{\ell}>0$, and take $h$ maximal so that $d_h>0$.  Note that by assumption, $M'\equiv EMA\ (q^e)$, where $E,A$ are necessarily invertible modulo $q$
Therefore $q^{\ell}\parallel M'$ as $q^{\ell}\parallel M$.
So for $0\le i<\ell$, we have $d_i'=0$, and $d_{\ell}'>0$.
We first want to show $d_i'=d_i$ for each $i$ with $\ell\le i\le e$.
 
If $\ell=e$ then, since $q^t|C$, we
have $M'\equiv M\equiv 0\ (q^e)$ so we are done.
So suppose $\ell<e$.  

Take $r=\min(h-\ell,e-1-\ell)$.
For $0\le i\le r$, take 
$$S_i=I_{d_{\ell}}\perp qI_{d_{\ell+1}}\perp\cdots\perp q^{i}I_{d_{\ell+i}}\perp q^{i}I\in\Z^{n,n}.$$
So $S_0=I$, and
$$q^{-\ell} S_i^{-1} M=M_{\ell}\perp M_{\ell+1}\perp\cdots\perp M_{\ell+i-1}
\perp q^{-\ell-i}\big( q^{\ell+i}M_{\ell+i}\perp\cdots\perp q^hM_h\big).$$

Suppose that $0\le i<r$, $S_i^{-1}ES_i$ is integral (hence invertible modulo $q$), and
$d_j'=d_j$ for $\ell\le j< \ell+i$.
We claim $d'_{\ell+i}=d_{\ell+i}$ and $S_{i+1}^{-1}ES_{i+1}$ is integral.
We have
\begin{align*}
d_{\ell}+d_{\ell+1}+\cdots+d_{\ell+i}&=\rank_q (q^{-\ell}S_i^{-1}M)\\
&=\rank_q(S_i^{-1}ES_i)(q^{-\ell}S_i^{-1}M)A.
\end{align*}
Since $S_i^{-1}ES_i$ and $A$ are integral and invertible modulo $q$, and 
$q^{-\ell}S_i^{-1}M$ is integral, we must have that 
$q^{-\ell}S_i^{-1}M'$
is integral.
Therefore
\begin{align*}
d_{\ell}+d_{\ell+1}+\cdots+d_{\ell+i}
&=\rank_q(q^{-\ell}S_i^{-1}M')\\
&=d_{\ell}+d_{\ell+1}+\cdots+d_{\ell+i-1}+d'_{\ell+i}.
\end{align*}
Hence $d_{\ell+i}=d'_{\ell+i}.$
Also, we have 
$$q^{-\ell}S_i^{-1}M=\begin{pmatrix} U_1\\&qU_2\end{pmatrix},\ q^{-\ell}S_i^{-1}M'=\begin{pmatrix} U'_1\\&qU'_2\end{pmatrix},$$
and
$$S_i^{-1}ES_i=\begin{pmatrix} E_1&E_2\\E_3&E_4\end{pmatrix},\ A=\begin{pmatrix} A_1&A_2\\A_3&A_4\end{pmatrix}$$
where $U_1, U_1', E_1, A_1$ are $(d_{\ell}+\cdots+d_{\ell+i})\times (d_{\ell}+\cdots+d_{\ell+i})$, and $U_1, U_1'$ are
invertible modulo $q$.  So (recalling that $EMA\equiv M'\ (q^e)$), we have
$$\begin{pmatrix} U_1'\\&0\end{pmatrix}\equiv\begin{pmatrix} E_1U_1A_1&E_1U_1A_2\\E_3U_1A_1&E_eU_1A_2\end{pmatrix}\ (q).$$
Hence $E_1, A_1$ are invertible modulo $q$ and $E_3\equiv0\ (q)$.  Thus with $c=d_{\ell}+\cdots+d_{\ell+i+1}$,
$$\begin{pmatrix} I_c\\&\frac{1}{q}I\end{pmatrix} S_i^{-1}ES_i\begin{pmatrix} I_c\\&qI\end{pmatrix}$$
is integral; that is, $S_{i+1}^{-1}ES_{i+1}$ is integral.

Hence by induction on $i$, we have that $d_{\ell+i}=d'_{\ell+i}$ for 
$0\le i\le r=\min(h-\ell,e-1-\ell)$, and
$S_r^{-1}ES_r$ is integral.  Since $M'\equiv \,^tAM\,^tE\ (q^t)$, the above argument also shows that $S_rAS_r^{-1}$ is integral.
%Further, we know $A\,^tD\equiv I\ (q)$ and $EMB+ED\equiv I\ (q)$.  

With renewed notation, write $E=(E_{ij})$, $A=(A_{ij})$  where
$E_{ij},A_{ij}$ are $d_i\times d_j$ ($\ell\le i,j\le h$).  
Since $E,A$ are invertible modulo $q$
and $S_r^{-1}ES_r, S_rAS_r^{-1}$ are integral, we have $E_{ij}, A_{ij}\equiv0\ (q^{i-j})$ whenever $j<i<e$.  
Thus $E_{ii}$ and $A_{ii}$ are invertible modulo $q$ for all $i$.  
Hence for $\ell\le i<e$, we have
$$q^iM_i'\equiv \sum_{j=\ell}^h E_{ij}q^jM_jA_{ji}\ (q^e).$$
For $j<i<e$ we have $q^jE_{ij}M_jA_{ji}\equiv 0\ (q^{i-j})$, and so 
$M_i'\equiv E_{ii}A_iA_{ii}\ (q).$
Since $E(M\ I)\begin{pmatrix}A&B\\0&D\end{pmatrix}\equiv (M'\ I)\ (q^e)$, 
we have
$E(MB+D)\equiv I\ (q^e)$.
We have $q^{\ell}|M$, so $ED\equiv I\ (q^{\ell})$.
We also have $\,^tAD\equiv I\ (q^e)$, so $E\equiv\,^tA\ (q^{\ell})$.
Therefore $E_{ii}\equiv\,^tA_{ii}\ (q^{\ell})$ for all $i$.
Hence for all $i$, we have 
$$\left(\frac{\det M_i'}{q}\right)=\left(\frac{\det M_i}{q}\right).$$
Since $\begin{pmatrix}I&0\\M&I\end{pmatrix}$ and $\begin{pmatrix}I&0\\M'&I\end{pmatrix}$
are reduced representatives modulo $q^e$, this means that 
$q^i M'_i\equiv q^i M_i\ (q^e)$
for $\ell<i< e$ and 
$q^{\ell}M_{\ell}\equiv q^{\ell}M'_{\ell}\ (q^e)$ if $\ell=0$ or $\ell=e$ or $0<\ell<h=e$.

So now suppose that $0<\ell\le h<e$.  Thus with $m=\min(\ell,e-h)$, we have
$M_\ell\equiv\big<1,\ldots,1,\varepsilon_{\ell}\big>\ (q^{e-\ell})$ and 
$M'_\ell\equiv\big<1,\ldots,1,\varepsilon'_{\ell}\big>\ (q^{e-\ell})$ where $1\le \varepsilon_{\ell},\varepsilon'_{\ell}\le q^{m}$,
$q\nmid\varepsilon_{\ell},\varepsilon'_{\ell}$.
Since 
$$q^{-\ell}S_{h-\ell}^{-1}EMA=(S_h^{-1}ES_{h-\ell})(q^{-\ell}S_{h-\ell}^{-1}M)A,$$
we have
$$M_{\ell}'\perp\cdots\perp M_h'
\equiv(S_{h-\ell}^{-1}ES_{h-\ell})(M_{\ell}\perp\cdots\perp M_h)A\ (q^{e-h})$$
with $S_{h-\ell}^{-1}ES_{h-\ell}$ integral with determinant 1, and $\det A\equiv 1\ (q^{\ell})$.
Hence
$$\det(M_{\ell}'\perp\cdots\perp M_h')\equiv \det(M_{\ell}\perp\cdots\perp M_h)\ (q^m).$$
We have seen that for $\ell<i\le h$ (with $d_i>0$), we have
$q^i M'_i\equiv q^i M_i\ (q^e)$, and hence $\det M'_i\equiv \det M_i\ (q^{e-i})$.  Thus $\det M'_i\equiv \det M_i\ (q^m)$ for $\ell<i\le h$.  Consequently, since $\det M_i$ is a unit modulo $q$ when $d_i>0$, we have $\det M'_{\ell}\equiv \det M_{\ell}\ (q^m)$.  Hence
$M_{\ell}'\equiv M_{\ell}\ (q^m)$, and so 
$M'\equiv M\ (q^e)$.

As this holds for all primes $q|\stufe$ with $q^e\parallel\stufe$, we have $M'\equiv M\ (q^e)$.
Thus 
$\gamma_{_{M'}}^{-1}\gamma_{_M}\equiv I\ (\stufe)$, and hence $\gamma_{_{M'}}\in\gamma_{_M}\Gamma(\stufe)$.
\end{proof}

\bigskip
\section{Evaluating average theta series at the cusps}

As noted earlier, in \cite{S35} Siegel showed that the  value of the average theta series at any 0-dimensional cusp is given by a generalized Gauss sum.  Here we first review that result, using the representatives for the 0-dimensional cusps that we described earlier.  Then we unwind the generalized Gauss sum to realize it explicitly in terms of powers of primes, Legendre symbols, and eighth roots of unity.

\begin{prop}  Suppose that 
$L=\Z x_1\oplus\cdots\oplus\Z x_m$ with $m=2k$, $k\in\Z_+$.  Also suppose that 
$Q\in\Z^{m,m}_{\sym}$ is the matrix for a positive definite, even integral quadratic form on $L$ relative to the given basis for $L$; let $\stufe$ be the level of $Q$.
Take $\begin{pmatrix}I&0\\-M&I\end{pmatrix}\in Sp_n(\Z)$.  Then 
%with $\theta^{(n)}(L;\tau)$ the degree $n$ Siegel theta series,
we have
$$\lim_{\tau\to i\infty}\theta^{(n)}(L;\tau)|\begin{pmatrix}I&0\\-M&I\end{pmatrix}
=\prod_{\substack{q^e\parallel\stufe\\q\,\text{prime}}} a_q(L,M)$$
where, with $q$ prime and $q^e\parallel\stufe$,
$$a_q(L,M)=q^{-emn}\sum_{V\in\Z^{m,n}/q^{e}\Z^{m,n}} \e\{\stufe^2 Q^{-1}(V)M/q^{2e}\}.$$

\end{prop}

\begin{proof}
We will use the Inversion Formula (Lemma 1.3.15 \cite{And}), which says the following.
With $U_0\in\Q^{m,n}$ and
$$\theta^{(n)}(L,U_0;\tau)=\sum_{U\in\Z^{m,n}}\e\{Q(U+U_0)\tau\},$$
we have
\begin{align*}
\theta^{(n)}(L,U_0;\tau)
&=(\det Q)^{-n/2}(\det(-i\tau))^{-m/2}
\sum_{U'\in\Z^{m,n}}\e\{-Q^{-1}(U')\tau^{-1}-2\,^tU'U_0\}.
\end{align*}

Take $\begin{pmatrix}I&0\\-M&I\end{pmatrix}\in Sp_n(\Z)$.  Then 
applying the Inversion Formula we have
\begin{align*}
&\theta^{(n)}(L;\tau(-M\tau+I)^{-1})\\
&\quad=
(\det Q)^{-n/2}(\det(-i\tau(-M\tau+I)^{-1}))^{-m/2}\\
&\qquad\cdot
\sum_{U\in\Z^{m,n}} \e\{-Q^{-1}(U)(-M+\tau^{-1})\}\\
% {&\quad=
% (\det Q)^{-n/2}(\det(-i\tau(-M\tau+I)^{-1}))^{-m/2}\\
% &\qquad\sum_{\substack{U_0\in\Z^{m,n}/\stufe\Z^{m,n}\\U\in\Z^{m,n}}}
% \e\{-Q^{-1}(U_0+\stufe U)(-M+\tau^{-1})\}\\
&\quad=
(\det Q)^{-n/2}(\det(-i\tau(-M\tau+I)^{-1}))^{-m/2}\\
&\qquad\sum_{U_0\in\Z^{m,n}/\stufe\Z^{m,n}}\e\{Q^{-1}(U_0)M\}
\theta^{(n)}(\stufe^2Q^{-1},\stufe^{-1}U_0;-\tau^{-1}).
\end{align*}
Applying the Inversion Formula again, we get
\begin{align*}
&\theta^{(n)}(L;\tau(-M\tau+I)^{-1})\\
% &\quad=
% (\det Q)^{-n/2}(\det(-i\tau(-M\tau+I)^{-1}))^{-m/2}\\
% &\qquad\cdot
% (\det\stufe^2 Q^{-1})^{-n/2}(\det(i\tau^{-1}))^{-m/2}\\
% &\qquad\cdot
% \sum_{U_0\,(\stufe)}\e\{Q^{-1}(U_0)M\}
% \sum_{U\in\Z^{m,n}}\e\{\stufe^2 Q^{-1}(U)\tau-2\stufe^{-1}\,^tU_0U\}\\}
&\quad=
\stufe^{-mn}\det(-M\tau+I)^k
\sum_{U_0, U_1\in\Z^{m,n}/\stufe\Z^{m,n}}\e\{Q^{-1}(U_0)M-2\stufe^{-1}\,^tU_0U_1\}\\
&\qquad \cdot\sum_{U\in\Z^{m,n}}\e\{Q(\stufe^{-1}U_1+U)\tau\}.
\end{align*}

Now we consider
\begin{align*}
&\lim_{\tau\to i\infty}\theta^{(n)}(L;\tau)|\begin{pmatrix}I&0\\-M&I\end{pmatrix}\\
&\quad=\stufe^{-mn}\sum_{U_0, U_1\,(\stufe)}\e\{Q^{-1}(U_0)M-2\stufe^{-1}\,^tU_0U_1\}\\
&\qquad \cdot \lim_{\tau\to i\infty}\sum_{U\in\Z^{m,n}}\e\{Q(\stufe^{-1}U_1+U)\tau\}.
\end{align*}
We have
$$\lim_{\tau\to i\infty}\sum_{U\in\Z^{m,n}}\e\{Q(\stufe^{-1}U_1+U)\tau\}
=\begin{cases}1&\text{if $U_1\in\stufe\Z^{m,n}$}\\0&\text{otherwise.}
\end{cases}$$
Hence
$$\lim_{\tau\to i\infty}\theta^{(n)}(L;\tau)|\begin{pmatrix}I&0\\-M&I\end{pmatrix}=\stufe^{-mn}\sum_{U_0\,(\stufe)}\e\{Q^{-1}(U_0)M\}.$$

Write $\stufe=q_1^{e_1}\cdots q_s^{e_s}$ where $q_1,\ldots,q_s$ are the distinct primes dividing $\stufe$, and set $\stufe_i=\stufe/q_i^{e_i}$. 
Let $\calL=\Z^{m,n}$ (an additive group).
One easily verifies that the map
$$(U_1+\stufe\calL,\ldots,U_s+\stufe\calL)\mapsto
U_1+\cdots+U_s+\stufe\calL$$
defines an isomorphism from 
$\stufe_1\calL\oplus\cdots\oplus\stufe_s\calL$ onto $\calL/\stufe\calL$.
Also, for $U_i=\stufe_iV_i\in\stufe_i\calL$ ($1\le i\le s$), since $\stufe Q^{-1}$ is even integral we have
$$Q^{-1}(U_1+\cdots+U_s)
\equiv\sum_{i=1}^s Q^{-1}(U_i)\ (2\Z).$$
Hence
$$\sum_{U\in\calL/\stufe\calL}\e\{Q^{-1}(U)M\}
=\prod_{i=1}^s\left(\sum_{V_i\in\calL/q_i^{e_i}\calL}\e\{(\stufe_i)^2 Q^{-1}(V_i)M\}\right).$$
\end{proof}

Next we use the local structure of $Q$ over $\Z_q$ for a prime $q|\stufe$ to simplify the sum defining $a_q(L,M)$, describing it in terms of invariants of $\Z_qL$, $M$ modulo $q^e$, and 
certain generalized Gauss sums, defined as follows.

\smallskip\noindent
{\bf Definition.}  Suppose that $q$ is prime, and $r,d,h\in\Z_+$.  Take $J'\in\Z^{r,r}_{\sym}$ and $M'\in\Z^{d,d}_{\sym}$ so that $J'$ and $M'$ are invertible modulo $q$, and $2|J'$ when $q\not=2$.  Set
$$\G_{J',M'}(q^h)=\sum_{x\in\Z^{r,d}/q^h\Z^{r,e}} \e\{J'(x)M'/q^h\}.$$
For $x,y\in\Z^{r,d}$, one easily verifies that $\e\{J'(x+q^hy)M'/q^h\}=\e\{J'(x)M'/q^h\},$ and hence $\G_{J',M'}(q^h)$ is well-defined.
Note that for $E\in SL_r(\Z)$ and $G\in SL_d(\Z)$, $EUG$ varies over $\Z^{r,d}/q^h\Z^{r,d}$ as $U$ does; hence with $J''=^tEJ'E$ and $M''=GM'\,^tG$, we have
$\G_{J'',M''}(q^h)=\G_{J',M'}(q^h).$  Also, $\G_{J',M'}(q^h)=G_{M',J'}(q^h).$
\smallskip

\begin{prop}  Suppose that $L=\Z x_1\oplus\cdots\Z x_m$ is equipped with an even integral quadratic form represented by $Q\in\Z^{m,m}_{\sym}$ relative to the given basis for $L$.
Let $\stufe$ be the level of $Q$, and suppose that $q$ is a prime with $q^e\parallel\stufe$ where $e\in\Z_+$.  
% When $q\not=2$, fix $\omega\in\Z$ so that $\left(\frac{\omega}{q}\right)=-1$.

\begin{enumerate}
\item[(a)] 
There is some $G\in SL_m(\Z_q)$ so that
$$GQ\,^tG\equiv J_0\perp qJ_1\perp\cdots q^eJ_e\ (q^{e+2})$$
where each $J_c$ is $r_c\times r_c$ for some $r_c$.  Further, when $q\not=2$ and $r_c>0$, 
$$J_c=2\big<1,\ldots,1,\nu_c\big>$$
with $q\nmid\nu_c$;
% with $\nu_c=1$ or $\omega$;
when $q=2$ and $r_c>0$,
$$J_c=\big<\mu_1,\ldots,\mu_{r_c}\big> \text{ or } 
H^{d_c/2} \text{ or } H^{d_c/2-1}\perp A_c$$
where $\mu_1,\ldots,\mu_{r_c}\in\{1,3,5,7\}$, $H=\begin{pmatrix}0&1\\1&0\end{pmatrix}$,
and $A_c=\begin{pmatrix}2a'_c&a_c\\a_c&2a'_ca_c^2\end{pmatrix}$
with $a'_c=0$ or 1, and $a_c$ odd.
Also, when $q=2$ and $r_0>0$, $J_0$ is even integral.
Further, when $q\not=2$, $r_e>0$;
when $q=2$, $J_0$ is even integral, and either $r_e>0$ and $J_e$ is even integral, or $r_e=0$ and  $r_{e-1}>0$ with $J_{e-1}$ diagonal.

\item[(b)]  Take $M\in\Z^{n,n}_{\sym}$ so that
$$M=M_0\perp qM_1\perp\cdots\perp q^eM_e$$
with each $M_j$ $d_j\times d_j$, and $M_j$ is invertible modulo $q$ when $d_j>0$.
Then 

\begin{align*}
&\sum_{V\in\Z^{m,n}/q^e\Z^{m,n}}\e\{\stufe^2 Q^{-1}(V)M/q^{2e}\}\\
&\quad
=q^{mne}\prod_{c=1}^e\prod_{j=0}^{c-1}q^{r_cd_j(j-c)}\G_{J_c',M_j}(q^{c-j})
\end{align*}
where, for each $c$ so that $r_c>0$, 
$J_c'=\begin{pmatrix}I\\&u_c\end{pmatrix}J_c\begin{pmatrix}I\\&u_c\end{pmatrix}$
for some $u_c$, $q\nmid u_c$
\end{enumerate}
\end{prop}

\begin{proof}  
(a) Fix a prime $q$ with $q^e\parallel\stufe$.
By \S93 of \cite{O'M}, or equivalently Theorems 8.5 and 8.9 of \cite{Ger}, we know that there is some $G'\in SL_m(\Z_q)$ so that 
$$^tG'QG'=J_0\perp qJ_1\perp\cdots q^eJ_e$$
where each $J_c$ is as in the statement of the proposition; in particular, each $J_c$ is invertible modulo $q$.  
Note that since $Q$ is even integral, when $q=2$ and $r_0>0$, we have that $J_0$ is even integral; also, when $q=2$ and $r_e>0$, we have that $J_e$ is even integral since $\stufe Q^{-1}$ is even integral.
Taking $G\in SL_m(\Z)$ so that $G\equiv G'\ (q^{e+1})$, we get
$$^tGQG\equiv J_0\perp qJ_1\perp\cdots q^eJ_e\ (q^{e+2}).$$

(b)  Take $G'$ as in (a).
Then
$$(\,^tG'QG')^{-1}=J_0^{-1}\perp\cdots\perp q^{-e}J_e^{-1}.$$
% When $q\not=2$, the definition of the level of $Q$ tells us that $s=e$ with $r_e>0$.
% When $q=2$, the definition of the level of $Q$ tells us that $s=e$ and either $r_e>0$ % and $J_e$ is even integral, or $r_e=0$ and $r_{e-1}>0$ with $J_{e-1}$ diagonal.
Take $\stufe'=\stufe/q^e$ and $u\in\Z_q$ so that
$$G''=(\stufe' G')^{-1}(J_0\perp J_1\perp\cdots\perp J_e)
\begin{pmatrix}I\\&u\end{pmatrix}\in SL_m(\Z_q)$$ 
(recall that each $J_i$  is invertible over $\Z_q$ whenever $r_i>0$).
Take $E\in SL_m(\Z)$ so that $E\equiv G''\ (q^e)$.  Thus
$q^e(\stufe')^2\,^tEQ^{-1}E\equiv Q'\ (q^e)$ where
$$Q'=q^eJ_0'\perp q^{e-1}J_1'\perp\cdots\perp J_e'\ (q^{e+1})$$
and, for each $c$ so that $r_c>0$, either $J_c'=J_c$ or
$J_c'=\begin{pmatrix}I\\&u\end{pmatrix}J_c\begin{pmatrix}I\\&u\end{pmatrix}.$
Since $EV$ varies over $\Z^{m,n}/q^e\Z^{m,n}$ as $V$ does, we have
\begin{align*}
&\sum_{V\in\Z^{m,n}/q^e\Z^{m,n}}\e\{\stufe^2Q^{-1}(V)M/q^{2e}\}\\
&\quad =
\sum_{V\in\Z^{m,n}/q^e\Z^{m,n}}\e\{Q'(V)M/q^e\}\\
&\quad=
\prod_{c=0}^e\prod_{j=0}^e\sum_{x\in\Z^{r_c,d_j}/q^e\Z^{r_c,d_j}}
\e\{J'_c(x)M_j/q^{c-j}\}\\
&\quad=
q^{r_0e(d_0+\cdots+d_e)}
\prod_{c=1}^e q^{r_ce(d_c+\cdots+d_e)}
\prod_{j=0}^{c-1}q^{r_cd_j(e-c+j)}\G_{G'_c,M_j}(q^{c-j})\\
&\quad=
q^{mne}\prod_{c=1}^e\prod_{j=0}^{c-1}q^{r_cd_j(j-c)}\G_{J'c,M_j}(q^{c-j})
\end{align*}
where, for the last equality, we used that $d_0+\cdots+d_e=n$ and $r_0+\cdots+r_e=m$.
\end{proof}

We now evaluate the Gauss sums that appear in the above proposition.
We first use a standard argument to reduce the modulus of the Gauss sum.

\begin{prop}  Let $q$ be a prime, $r,d\in\Z_+$, $J'\in\Z^{r,r}_{\sym}$, $M'\in\Z^{d,d}_{\sym}$ ($r,d>0$) so that $J', M'$ are invertible modulo $q$.
Then for $h\ge 2$, we have
$$\G_{J',M'}(q^h)=
\begin{cases}
q^{rdh/2}&\text{if $2|h$,}\\
q^{rd(h-1)/2}\G_{J',M'}(q)&\text{otherwise.}
\end{cases}$$
\end{prop}

\begin{proof}  
We have
\begin{align*}
\G_{J',M'}(q^h)
&=\sum_{x\in\Z^{r,d}/q^{h-1}\Z^{r,d}}
\sum_{y\in\Z^{r,d}/q\Z^{r,d}}
\e\{J'(x+q^{h-1}y)M'/q^h\}\\
&=\sum_{x\in\Z^{r,d}/q^{h-1}\Z^{r,d}} \e\{J'(x)M'/q^h\}
\sum_{y\in\Z^{r,d}/q\Z^{r,d}} \e\{M'\,^txJ'y/q\}.
\end{align*}
This last sum on $y$ is a character sum, yielding 0 if $q\nmid x$ and $q^{rd}$ otherwise.
Hence
$\G_{J',M'}(q^h)=q^{rd}\G_{J',M'}(q^{h-2}).$
Repeated applications of this identity yields the result.
\end{proof}

Now we evaluate the Gauss sums $\G_{J',M'}(q)$, separating the cases of $q$ odd and even.

\begin{prop}  Let $q$ be an odd prime.  Suppose that
$J'\in\Z^{r,r}_{\sym}$ and $M'\in\Z^{d,d}_{\sym}$ 
($r,d>0$) so that $J',M'$ are invertible modulo $q$.
 Then
$$\G_{J',M'}(q)
=\left(\frac{\det J}{q}\right)^d \left(\frac{\det M'}{q}\right)^r
\left(\G_1(q)\right)^{rd}$$
where $\G_1(q)$ is the classical Gauss sum.
Thus for $u\in\Z$ with $q\nmid u$ and 
$$J=\begin{pmatrix}I\\&u\end{pmatrix}J'\begin{pmatrix}I\\&u\end{pmatrix},$$
we have
$\G_{J',M'}(q)=\G_{J,M'}(q).$
\end{prop}

\begin{proof} 
As in the proof of Proposition 5.2, we can find $E\in SL_r(\Z)$ and $E'\in SL_d(\Z)$ so that $^tEJ'E\equiv2\big<1,\ldots,1,\nu\big>\ (q)$ and $^tE'M'E'\equiv \big<1,\ldots,1,\varepsilon\big>\ (q).$  As $ExE'$ varies over $\Z^{r,d}/q\Z^{r,d}$ as $x$ does, we can replace $x$ by $ExE'$ in the sum defining
$\G_{J',M'}(q)$.
Expanding $^tx(\,^tEJ'E)x(\,^tE'M'E')$ 
% and using that $\e\{*\}=\exp(\pi iTr(*))$, 
we find that
\begin{align*}
\G_{J',M'}(q)
&=\left(\G_{1,M'}(q)\right)^{r-1}\cdot\G_{\nu,M'}(q)\\
&=\left(\G_{1}(q)\right)^{(r-1)(d-1)}\left(\G_{\nu}(q)\right)^{d-1}\left(\G_{\varepsilon}(q)\right)^{r-1}\cdot\G_{\nu\varepsilon}(q).
\end{align*}
Since $\G_a(q)=\left(\frac{a}{q}\right)\G_1(q)$, the result follows.
\end{proof}

% For the next proposition, note that an even integral $2\times 2$ matrix that is invertible % modulo 2 is of the form $\begin{pmatrix}2a&b\\b&2c\end{pmatrix}$ 
% where $b$ is odd.
To help us state the next proposition, we introduce the following terminology.

\smallskip\noindent{\bf Definition.}  Suppose that $J'\in\Z^{r,r}_{\sym}$ ($r>0$) so that $2\nmid\det J'$ and $J'$ is even integral.  As discussed in the proof of Proposition 5.2, we can find $E\in SL_r(\Z)$ so that 
$$^tEJ'E\equiv H\perp\cdots\perp H\perp A'\ (4)$$
where $H=\begin{pmatrix}0&1\\1&0\end{pmatrix}$ and $A'=\pm H$ or
$\pm\begin{pmatrix}2&1\\1&2\end{pmatrix}.$  When $A'=\pm H$ then we say that $J'$ is hyperbolic modulo $4$; note that $J'$ is hyperbolic modulo 4 exactly when $J'$ is even integral and $(-1)^{r/2}\det J'\equiv 1\ (4).$
\smallskip

\begin{prop}  Suppose that $J'\in\Z^{r,r}_{\sym}$ and $M'\in\Z^{d,d}_{\sym}$ ($r,d>0$)
so that $J', M'$ are invertible modulo 2.
\begin{enumerate}
\item[(a)]  Suppose that $J'$ and $M'$ are even integral.  Then $\G_{J',M'}(2)=2^{rd/2}.$
\item[(b)]  Suppose that either $J'$ or $M'$ is not even integral, and that the other is even integral and
hyperbolic modulo 4.  Then Then $\G_{J',M'}(2)=2^{rd/2}.$
\item[(c)]  Suppose that either $J'$ or $M'$ is not even integral, and that the other is 
even integral but not hyperbolic modulo 4.  Then $\G_{J',M'}(2)=(-1)^{rd}2^{rd/2}.$
\item[(d)]  Suppose that neither $J'$ nor $M'$ is even integral; in this case there exist
$E\in SL_r(\Z)$, $E'\in SL_d(\Z)$, $r',d'\in\Z$ so that
$$^tEJ'E\equiv I_{r'}\perp 3I_{r-r'}\ (4) \text{ and }^tE'M'E'\equiv I_{d'}\perp 3I_{d-d'}\ (4).$$
Then $\G_{J',M'}(2)=(2i)^{rd/2}(-i)^{2r'd'-rd'-r'd}.$
\end{enumerate}
Also, for odd $u\in\Z$ and $$J=\begin{pmatrix}I\\&u\end{pmatrix}J'\begin{pmatrix}I\\&u\end{pmatrix},$$ we have
$\G_{J,M'}(2)=\G_{J',M'}(2).$
\end{prop}

\begin{proof}
As we saw in the proof of Proposition 5.2, we can find $E\in SL_r(\Z)$ 
so that when $J'$ is even integral we have
$$^tEJ'E\equiv H\perp\cdots\perp H\perp A'\ (4)$$ with $A'=\pm H$ or $\pm\begin{pmatrix}2&1\\1&2\end{pmatrix}$, and when $J'$ is not even integral we have
$^tEJ'E\equiv I_{r'}\perp 3I_{r-r'}\ (4)$ for some $r'$.  Similarly, we can find $E'\in SL_d(\Z)$ so that when $M'$ is even integral we have
$$^tE'M'E'\equiv H\perp\cdots\perp H\perp A''\ (4)$$ with $A''=\pm H$ or $\pm\begin{pmatrix}2&1\\1&2\end{pmatrix}$, and when $M'$ is not even integral we have
$^tE'M'E'\equiv I_{d'}\perp 3I_{d-d'}\ (4)$ for some $d'$.
In the sum defining $\G_{J',M'}(2)$, we can replace $x\in\Z^{r,d}/2\Z^{r,d}$ by $ExE'$, and then
 $\G_{J',M'}(2)$ decomposes as a product of sums over $2\times 2$ or $2\times 1$ or $1\times 2$ or $1\times 1$ matrices modulo $2$.

For $A'=\begin{pmatrix}2a'&b'\\b'&2c'\end{pmatrix}$, $A''=\begin{pmatrix}2a''&b''\\b''&2c''\end{pmatrix}$
with $b',b''$ odd, we have
\begin{align*}
\sum_{x\in\Z^{2,2}/2\Z^{2,2}} \e\{A'(x)A''/2\}
&=2\sum_{u,u'\in\Z/2\Z} \e\{uu'b'b''\}
=4.
\end{align*}
With $A'=\pm H$,
$A''=\pm\begin{pmatrix}2&1\\1&2\end{pmatrix}$ and $\varepsilon$ odd,
we have
$$\sum_{x\in\Z^{2,1}} \e\{A'(x)\varepsilon/2\}=2
\text{ and }
\sum_{x\in\Z^{2,1}} \e\{A''(x)\varepsilon/2\}=-2.$$
Finally, with $\nu\varepsilon$ odd, we have
$$\sum_{x\in\Z/2\Z} \e\{\nu\varepsilon x^2/2\}
=\begin{cases}
1+i&\text{if $\nu\varepsilon\equiv 1\ (4)$,}\\
1-i&\text{if $\nu\varepsilon\equiv -1\ (4)$.}
\end{cases}
$$
From this the proposition follows.
\end{proof}

\bigskip

\section{Proof of Theorem 1.1}
\smallskip

We have a dimension $2k$ $\Z$-lattice $L$ equipped with a positive definite, even integral quadratic form represented by $Q\in\Z^{2k,2k}_{\sym}$.
We let $M\in\Z^{n,n}_{\sym}$ vary so that 
$\{\E_{\gamma_{_M}}\}$ is a basis for the space of Siegel Eisenstein series
of degree $n$, weight $k$, level $\stufe$, and character $\chi_L$
(where $\chi_L$ is the character associated to $\theta^{(n)}(L;\tau)$,
as defined in Section 2).
By Proposition 4.2, we can assume that 
each
$\gamma_{_M}=\begin{pmatrix}I&0\\M&I\end{pmatrix}$ is a reduced representative modulo $\stufe/2^{\ord_2(\stufe)}$ and a partially reduced representative modulo $2^{\ord_2(\stufe)}$.
From \cite{S35}, we know that for some $a'(L,M)$
we have
$$\theta^{(n)}(\gen L;\tau)=\sum_M a'(L,M)\E_{\gamma_{_M}}.$$
For $\E_{\gamma_{_M}}$ and $\E_{\gamma_{_N}}$ in the basis for Siegel Eisenstein series, Proposition 3.1 gives us
$$\lim_{\tau\to i\infty}\E_{\gamma_{_N}}(\tau)|\gamma_{_M}^{-1}
=\begin{cases}2&\text{if $N=M$ and $\stufe\le 2$,}\\
1&\text{if $N=M$ and $\stufe>2$,}\\
0&\text{otherwise,}\end{cases}$$
and Proposition 5.1 gives us
$$\lim_{\tau\to i\infty}\theta^{(n)}(L;\tau)|\gamma_{_M}^{-1}=
\prod_{q|\stufe}a_q(L,M)$$
where $a_q(L,M)$ is defined in Proposition 5.1.

Fix a prime $q|\stufe$ with $q^e\parallel\stufe$.
By Proposition 5.2, there is some $G\in SL_m(\Z)$ so that 
$$^tGQG\equiv J_0\perp qJ_1\perp\cdots\perp q^eJ_e\ (q^{e+2})$$
where for $0\le c\le e$, there is some $r_c$ so that  $J_c$ is $r_c\times r_c$, integral and symmetric, with $q\nmid\det J_c$ when $r_c>0$.
Also, by our choices of $M$, 
$$M\equiv M_0\perp qM_1\perp\cdots\perp q^eM_e\ (q^e)$$
where for $0\le j<e$, there is some $d_j$ so that
$M_j$ is $d_j\times d_j$, integral and symmetric, with $q\nmid M_j$ when $d_j>0$.
By Propositions 5.2 and 5.3, $a_q(L,M)$ is determined by $\G_{J_c',M_j}(q)$
($0\le j<c\le e$) where 
$J_c'=\begin{pmatrix}I\\&u_c\end{pmatrix}J_c\begin{pmatrix}I\\&u_c\end{pmatrix}$
for some $u_c$ with $q\nmid u_c$.  By Propositions 5.4 and 5.5, we have
$\G_{J'_c,M_j}(q)=\G_{J_c,M_j}(q)$, and when $q$ is odd, $\G_{J_c,M_j}(q)$ is determined by the dimensions and determinants of $J_c$ and $M_j$.

Now consider the case that $q=2$.  A priori, by Proposition 5.5, $\G_{J_c,M_j}(2)$ is determined by the structures of $J_c$ and $M_j$ modulo 4, yet we only know the structure of $M_{e-1}$ modulo 2, and we do not even know the structure of $M_e$ modulo 2.  However, the formula for $a_q(L,M)$ only involves the Gauss sums
$\G_{J_c,M_j}(2)$ for $0\le j<c\le e$, so we only need to ascertain that $\G_{J_e,M_{e-1}}(2)$ is well-determined by $M_{e-1}$ modulo 2 in the case that $r_e,d_{e-1}>0$.  In the case that $r_e,d_{e-1}>0$, we know from Proposition 5.2 that $J_e$ is even integral, and so by Proposition 5.5,
the value of $\G_{J_e,M_{e-1}}(2)$ is determined by whether $J_e$ is hyperbolic, and whether $M_{e-1}$ is even integral (which can be discerned by $M_{e-1}$ modulo 2).

Consequently Propositions 5.1--5.5 show that $\lim_{\tau\to i\infty}\theta^{(n)}(L;\tau)|\gamma_{_M}^{-1}$ is determined by $M$ and the local structure of $L$ at each prime dividing $\stufe$.  Hence
$$\lim_{\tau\to i\infty}\theta^{(n)}(\gen L;\tau)|\gamma_{_M}^{-1}
=\lim_{\tau\to i\infty}\theta^{(n)}(L;\tau)=\kappa\prod_{q|\stufe} a_q(L,M)$$
where $\kappa=1$ if $\stufe> 2$ and $\kappa=1/2$ otherwise.  Also, Propositions 5.1--5.5 give us the exact value of $a_q(L,M)$ for each prime $q|\stufe$.
Note that since $a'(L,M)\not=0$ for those $M$ in the theorem, we conclude that $\E_{\gamma_{_M}}\not=0$.

\end{document}